\documentclass[12pt]{amsart}
\usepackage{amssymb,amsmath,amsthm,comment}
\usepackage[left]{lineno}
\usepackage{blindtext}

\def\Q{\mathbb Q}

\def\Z{\mathbb Z}
\def\V{V}

\def\kq{K((t^{\Q}))}

\def\lm{{\rm{lm}}}
\def\lt{{\rm{lt}}}
\def\lexp{v}
\def\supp{{\rm Supp}}

\def\note[#1]{#1}

\newtheorem{theorem}{Theorem}[section]
\newtheorem{lemma}[theorem]{Lemma}
\newtheorem{prop}[theorem]{Proposition}
\newtheorem{cor}[theorem]{Corollary}

\theoremstyle{definition}
\newtheorem{definition}[theorem]{Definition}
\newtheorem{example}[theorem]{Example}

\numberwithin{equation}{section}

\theoremstyle{remark}

\numberwithin{equation}{section}

\begin{document}

\title[Valuations on Rational Function Fields in Two Variables]
{Reversely Well-Ordered Valuations on Rational Function Fields in Two Variables}

\author{Edward Mosteig}
\address{Department of Mathematics, Loyola Marymount University, Los Angeles, CA 90045}
\email{emosteig@lmu.edu}

\begin{abstract}
We  examine valuations on a rational function field $K(x,y)$ and analyze their behavior when restricting to an underlying polynomial ring $K[x,y]$. 
Motivated to solve the ideal membership problem in polynomial rings using Moss Sweedler's framework of generalized Gr\"obner bases, we produce an infinite collection of valuations $\lexp: K(x,y) \to \Z \oplus \Z$ such that $\lexp(K[x,y]^*)$ is reversely well-ordered. In addition, we construct a surprising example  where $\lexp(K[x,y]^*)$  is nonpositive, yet  not reversely well ordered.
\end{abstract}

\maketitle


\section{Introduction} \label{intro}
\setcounter{equation}{0}

Given an abelian group $\Gamma$,  a subset $S \subset \Gamma$ and  $\alpha \in \Gamma$ , we define $\alpha S =  S\alpha = \{ \alpha s \mid s \in S\}$ and $\alpha+S = S+\alpha = \{\alpha+s \mid s \in S\}$.
Whenever $R$ is a monoid, written additively, we denote by
 $R^*$
the nonzero elements of $R$.  (This applies in particular to the additive group of a ring.)
Given monoids $M$ and $N$ contained in an abelian group $\Gamma$,  we define  $M + N =\{m +n \mid m \in M, n \in N\}$ and $M-N = \{ m-n \mid m \in M, n \in N \}$.

In the 1980s, in order to develop an alternative method of
solving the ideal membership problem in polynomial rings, Moss Sweedler produced a generalization of Gr\"obner bases where monomial orders were replaced by valuations.
 The
fundamental idea is that monomial orders are well orderings on the set
of monomials, which leads to a natural reduction process using multivariate polynomial division.   Valuations, however,
permit a more general reduction process than provided by monomial orders.  The development of this theory can be found entirely in  Sweedler's unpublished
manuscript \cite{sweedler}.

\begin{definition}\label{def:valuation}
Given a field extension $L\mid K$ and a totally ordered abelian group $(\Gamma, +, <)$, we say that
 $$v: L \to \Gamma \cup \{\infty\}$$ is a $K$-{\bf valuation   on $L$} if for all $f,g \in L,$ the following hold:
\begin{enumerate}
\item[(i)] $\lexp(f) = \infty$ if and only if $f=0$;
\item[(ii)] $\lexp(fg) = \lexp(f) + \lexp(g)$;
\item[(iii)] $ \lexp(f+g) \ge \min\{ \lexp(f), \lexp(g) \}$;
\item[(iv)] If $\lexp(f) = \lexp(g) \neq \infty$, then
$\exists!\lambda\in {K}$ such that  $\lexp(f+\lambda g) > \lexp(f)$.
\end{enumerate}
From the axioms above, the strong triangle inequality follows: $ \lexp(f+g) = \min\{ \lexp(f), \lexp(g) \}$
whenever $\lexp(f) \neq \lexp(g)$.
Note that condition (iv) means that not only is $\lexp$  trivial on $K$, but $K$ is a field of
representatives for the residue field ${\mathcal O}_v/{\mathcal M}_v$, where ${\mathcal O}_v$ is the valuation ring ${\mathcal O}_v = \{ f \in L^* \mid v(f)  \ge 0\}$ with maximal ideal ${\mathcal M}_v = \{ f \in L^* \mid v(f)  > 0\}$.
When $\Gamma \cong \Z$, we call $v$ a {\bf discrete valuation of rank 1}. If $A$ is a domain such that $A \subseteq L$, we say that $v\mid_A$ is discrete of rank 1 whenever the $v$-image of the set of  nonzero elements of the field of fractions of $A$ is isomorphic to $\Z$.
\end{definition}

Although  mathematicians have analyzed the restriction of valuations to polynomial rings, historically the focus has almost been entirely on the case when $\lexp$ is {\em nonnegative} on the polynomial ring. However, our research has a distinctly different flavor since we require valuations to be {\em nonpositive} on the polynomial ring. One notable exception is in the area of order domains, where such valuations have been investigated in papers, such as \cite{gepe} and \cite{os}, with the purpose of developing algorithms for algebraic geometry codes.

In order to use
the algorithms constructed by Sweedler in \cite{sweedler} to solve the ideal membership problem in the polynomial ring $K[x_1, \dots, x_n]$, we must consider $K$-valuations on the rational function field $L=K(x_1, \dots, x_n)$.
We say that a partially ordered set is {\bf reversely well-ordered} if every nonempty subset  has a largest element.
A nontrivial $K$-valuation $\lexp$
on $K(x_1, \dots, x_n)$ can be used in Sweedler's generalized theory of Gr\"obner bases provided that $\lexp(K[x_1, \dots, x_n]^*)$ is reversely well-ordered.
From this, it trivially follows that $\lexp(K[x_1, \dots, x_n]^*)$ must be nonpositive.

For the entirety of this paper, we focus on the polynomial ring $K[x,y]$ in two variables over a field $K$ of arbitrary characteristic. Constructions of $K$-valuations $\lexp$ on $K(x,y)$ with $\lexp(K[x,y]^*)$ reversely well-ordered are investigated  in 
\cite{ms1}, \cite{ms2}, \cite{mo1} and \cite{mo2}.

We define $\kq$  to be the set of all maps $z:\Q \to K$ such that  $\supp(z)= \{e \in \Q : z(e) \neq 0 \}$ is well-ordered. 
 The set of such maps,  which we call Hahn power series,  was shown in \cite{hahn} to form a field in which addition is defined pointwise and multiplication is defined via convolution; i.e., if $z,u \in \kq$ and $i \in \Q$, then $(z + u)(i) = z(i) + u(i)$ and $(zu)(i) = \sum_{j+k=i} z(j)u(k)$.
To justify the name `series', we often use the notation $z  = \sum_{e \in \supp(z)} z(e) t^e$. 
 Exploiting this notation, we see there is a natural embedding $K(t) \hookrightarrow \kq$ where $t$ is sent to $t^1$, which is a series consisting of exactly one term.
 For each $z \in \kq$ that is transcendental over $K(t)$, there is a corresponding valuation $\lexp$ given by
 $$\begin{array}{rcl}
  K(x,y) &\rightarrow & \Q \\
 f & \mapsto & \min \supp(\varphi_z(f)) \\
 \end{array}$$
where $\varphi_z: K(x,y)^*  \to  \kq$ is the unique $K$-homomorphism such that $\varphi(x) = t^{-1}$ and $\varphi(y) = z$.
 The image $\lexp(K[x,y]^*)$ is often quite complex and poorly behaved. In some cases, $\lexp(K[x,y]^*)$ is non-positive and yet not reversely well-ordered. 
Such examples  in \cite{mo2} depend partially on  \cite{ked1} when $K$ has positive characteristic.  However, in \cite{ked2}, Kedlaya constructed a counterexample to a theorem appearing in \cite{ked1} and proceeds to produce a corrected version. Fortunately, Kedlaya also demonstrates in \cite{ked2} that the results in \cite{mo2} remain unaffected by this change.

Throughout the remainder of this paper, we turn our attention to the case when the value group is $\Z \oplus \Z$.
 One would expect the behavior of such valuations to be comparatively tame, though we provide an example at the end of this paper that suggests the situation is much more interesting than one would naively predict.

In Section \ref{section:growth}, we consider a $K$-valuation $\lexp$ on $L$ where $A$ is a domain  and $V$ is a $K$-vector space such that $K \subseteq A \subseteq V \subseteq L$. The images of the restriction of $\lexp$ to domains and vector spaces are monoids, and in order to properly study these objects, we first define the notion of a quotient monoid.

Given a submonoid $M$ of a commutative monoid $N$, we define an equivalence
relation on $N$ by setting $n_1 \sim_M n_2$ if and only if there exist $m_1 ,
m_2 \in M$ such
that $m_1 + n_1 = m_2 + n_2$ . Denote by $N/M$ the collection of all equivalence
classes under this relation, and for $n \in N$, let $\overline{n}$ denote the equivalence class
containing $n$.
We define a quotient map from $N$ to $N/M$ that sends $n$ to $\overline{n}$. The
set $N/M$ has
an additive monoid structure, called the quotient monoid of $N$ with respect to
$M$,
where we define $\overline{n_1} + \overline{n_2} = \overline{n_1 + n_2}$.  In general, if $M$ is understood based
on context, then we  write $\overline{N}$ in place of $N/M$.

We observe in Lemma \ref{lemma:digging} that given $c \in L$, there is at most one element of
$\lexp((V+Ac)^*)/\lexp(A^*)$ that is not in $\lexp(V^*)/\lexp(A^*)$.
In Theorem \ref{theorem:quotientplusone}, we provide sufficient conditions for this bound to be tight.

In Section \ref{valpoly}, we use the results of Section \ref{section:growth} to study $K$-valuations on $K(x,y)$ with value group $\Z \oplus \Z$.

In Section \ref{section:construction}, we demonstrate how to construct a class of $K$-valuations on $K(x,y)$ using representations of polynomials as linear combination of powers of a fixed element of $K(x)[y]$.
  This work provides a concrete method for  computing of the image of an arbitrary element of $K[x,y]^*$.

In Example \ref{ex:suitablevals} of Section \ref{examples}, we produce a collection of $K$-valuations $\lexp: K(x,y) \to \Z \oplus \Z$ with $\lexp(K[x,y]^*)$ reversely well ordered. However, even with value group $\Z \oplus \Z$, it is possible that the set $\lexp(K[x,y]^*)$ can be poorly behaved.
We demonstrate this by constructing a $K$-valuation  $\lexp: K[x,y]^* \to \Z \oplus \Z$ such
that $\lexp(K[x,y]^*)$ is nonpositive and yet not finitely-generated. In fact, $\lexp(K[x,y]^*)$ is  not even reversely well ordered in this final example.

\section{Bounds on the Growth of Valuations}\label{section:growth}

In this section, we investigate the collection of $\lexp$-images when extending a valuation $\lexp$ from one vector space to another.  In particular, we examine the codimensions of a chain of $K$-vector spaces 
$$V_1 \subseteq V_2 \subseteq V_3 \subseteq  \cdots $$
where each is contained in a  field $L$ endowed with a nontrivial $K$-valuation.
We are particularly interested in the case when the $v$-images of these vector spaces are reversely well ordered. 

Throughout this section,  we assume that $\lexp$ is a $K$-valuation on $L$ such that
$$K \subseteq A \subseteq V \subseteq L,$$ where  $A$ is a domain and $V$ is a $K$-vector space.
Note that since $A$ and $L$ both contain $K$, they are also $K$-vector spaces.

\begin{lemma}\label{lemma:digging}
Given $c \in L$, there is at most one element of
$\lexp((V+Ac)^*)/\lexp(A^*)$ that is not in $\lexp(V^*)/\lexp(A^*)$.
\end{lemma}

\begin{proof}
Given  $a_1,a_2 \in A^*, v_1,v_2 \in V$ such
that
\begin{equation}\label{eq:dig-at-most}
\overline{\lexp(v_1+a_1c)}, \overline{\lexp(v_2+a_2c)} \not\in \lexp(V^*)/\lexp(A^*),
\end{equation}
we must show that $\overline{\lexp(v_1+a_1c)}
 = \overline{\lexp(v_2+a_2c)}$.  If $\lexp(a_2(v_1+a_1c)) = \lexp(a_1(v_2+a_2c))$, then
$ \lexp(a_2)  + \lexp(v_1+a_1c) =  \lexp(a_1) + \lexp(v_2+a_2c) $ and so
$\overline{\lexp(v_1+a_1c)} = \overline{\lexp(v_2+a_2c)}$.  Thus, we need only consider the
possibility  $\lexp(a_2(v_1+a_1c)) \neq \lexp(a_1(v_2+a_2c))$.  

Define
$w = a_2(v_1 + a_1c)-a_1 (v_2+a_2c) = a_2v_1-a_1v_2 \in V$, in which case
$\overline{\lexp(w)} \in \lexp(V^*)/\lexp(A^*)$.
 By the strong
triangle inequality,
$$\lexp(w) = \min\{\lexp(a_2(v_1+a_1c)),
\lexp(a_1(v_2+a_2c))\}.$$
Suppose that $\lexp(w) = \lexp(a_2(v_1+a_1c))$, in which case
$\overline{\lexp(w)} =
\overline{\lexp(v_1+a_1c)}$.  However, $\overline{\lexp(v_1+a_1c)} \not\in
\lexp(V^*)/\lexp(A^*)$ by (\ref{eq:dig-at-most}), which contradicts the
statement that $\overline{\lexp(w)} \in \lexp(V^*)/\lexp(A^*)$.  The conclusion follows
similarly if we consider the case $\lexp(w) = \lexp(a_1(v_2+a_2c))$.
\end{proof}

Next we extend  a vector space by adding more than one basis element.

\begin{lemma}\label{prop:digging-card}
If $\dim_{K}V/A < \infty$, then $\lexp(V^*)/\lexp(A^*)$ has cardinality at most $\dim_{K}V/A$.
\end{lemma}

\begin{proof}
Write \begin{equation*}
V= Au_1 + \cdots + Au_d
\end{equation*}
where $d = \dim_K V/A$ and $u_1, \dots, u_d \in V$.
Define  $V_i = Au_1 + \cdots + Au_i$ for $1 \le i \le d$.  By
Lemma \ref{lemma:digging}, for each index $i$ such that $0 \le i \le d-1$,
\begin{equation*}
\# \lexp(V_{i+1}^*)/\lexp(A^*)  =   \# \lexp((V_i + Au_{i+1})^*)/\lexp(A^*) \le \#\lexp(V_i^*)/\lexp(A^*) + 1,
\end{equation*}
and so by induction, it follows that 
$$ \# \lexp(V_d^*)/\lexp(A^*) \le \# \lexp(V_{1}^*)/\lexp(A^*) + (d-1).$$
Since $\lexp(V_1^*)/\lexp(A^*)$ has cardinality 1, it follows that
$\# \lexp(V^*)/\lexp(A^*) \le d$.
\end{proof}

Our goal is to construct conditions whereby the 
bound produced in  Lemma  \ref{prop:digging-card} is tight. To this end, we first justify a few supporting lemmas.

\begin{lemma}\label{lemma:numericalsemigroup}
Suppose $\lexp(A^*)$ is nonpositive and $\lexp(A^*) - \lexp(A^*) = \Z \alpha$ 
where $\alpha$  is positive.
Then there exists $n\in\Z$ such that
 $k \alpha \in
\lexp(A^*)$ for all $k \le n$.
\end{lemma}

\begin{proof}
Since $\lexp(A^*)$ is nonpositive, there exists $h \in A^*$ such that $\lexp(h) = -m \alpha$ for some positive integer $m$.
 If $m =1$, then $\lexp(A^*) = \alpha \Z_{\le 0}$, from which the conclusion follows. Thus, we only need to consider the case $m>1$.

For each index $r$ such that $1 \le r \le m-1$,
select $f_r, g_r\in A^*$ such that $\lexp(f_r/g_r) = -r \alpha$. If we define $h_0 = g_1g_2 \cdots g_{m-1}$, note that $\lexp(h_0) = n \alpha$ where $n$ is a negative integer. Define $h_r = (h_0/g_r)f_r \in A^*$ for $1 \le r \le m-1$. For any such index $r$,  we observe that $\lexp(h_r) =  
(n-r) \alpha$.

Given $k \le n$, we can write  $n-k = qm+r$ where $q, r \in \Z_{\ge 0}$ such that $0 \le r \le m-1$. Then
$k \alpha = (n-qm-r) \alpha = q \lexp(h) + \lexp(h_r) = \lexp(h^qh_r)\in \lexp(A^*)$.
\end{proof}

\begin{lemma}\label{lemma:finite-difference-prep}
If $v|_A$ is  discrete of rank 1  and $\lexp(A^*)$ is nonpositive, then for all $\beta \in \lexp(V^*)$, every reversely well-ordered subset of
\begin{equation}\label{eq:b-a-big-prep}
(\beta + \lexp(A^*) -
\lexp(A^*))  \setminus \lexp(V^*) \end{equation}
is finite.
\end{lemma}

\begin{proof}
Suppose, towards contradiction, there exists $\beta \in \lexp(V^*)$ and a reversely well-ordered subset $R$ of $\lexp(L^*)$ such that 
\begin{equation}\label{eq:b-a-big-prep}
R \cap (\beta + \lexp(A^*) -
\lexp(A^*))  \setminus \lexp(V^*) \end{equation}
is infinite. Since $\lexp|A$ is  discrete of rank 1,  we can write $\lexp(A^*) - \lexp(A^*) = \Z \alpha,$ 
where $\alpha$ is chosen to be positive. From this, 
(\ref{eq:b-a-big-prep}) can be re-written as
\begin{equation}\label{eq:diffab}
 R \cap (\beta+ \Z \alpha)  \setminus \lexp(V^*).
\end{equation}

By Lemma \ref{lemma:numericalsemigroup}, there exists $n\in\Z$ such that
 $k \alpha \in
\lexp(A^*) \subseteq \lexp(V^*)$ for all $k \le n$.
For every such $k$, we have $\beta + k\alpha \in  \lexp(V^*)$.
  Therefore, expression (\ref{eq:diffab}) can be
written as
\begin{equation} \label{eq:b+z-alpha}
R \cap \{ \beta + k \alpha \mid k >  n  \}  \setminus \lexp(V^*).
\end{equation}
If this set had infinite cardinality, then since $\alpha$ is positive, there would be an infinite chain of
inequalities of the form
\begin{equation}
\beta + k_1 \alpha < \beta + k_2 \alpha < \beta + k_3 \alpha < \cdots
\end{equation}
Since these are all elements of the reversely well-ordered set $R$,  we have a
contradiction.
\end{proof}

\begin{lemma}\label{lemma:finite-difference}
Suppose the vector space extension $A \subseteq V$ is finite. If $v|_A$ is  discrete of rank 1  and $\lexp(A^*)$ is nonpositive, then every reversely well-ordered subset of
\begin{equation}\label{eq:b-a}
(\lexp(V^*)-\lexp(A^*)) \setminus \lexp(V^*)
\end{equation}
has finite cardinality.
\end{lemma}

\begin{proof}
By Lemma \ref{prop:digging-card}, the cardinality of $\lexp(V^*)/
\lexp(A^*)$ is finite, and so we can write 
\begin{equation}\label{eq:b-a-finite}
\lexp(V^*)/
\lexp(A^*) = \{ \overline{\lexp(u_1)}, \dots, \overline{\lexp(u_j)} \}
\end{equation}
where  $u_1, \dots, u_j \in V$.  Next, we claim that
\begin{equation}\label{eq:b-a-union}
\lexp(V^*) - \lexp(A^*) =  \bigcup_{i=1}^j  \lexp(u_i) + \lexp(A^*) -
\lexp(A^*). \end{equation}
Indeed, consider $\lexp(u) - \lexp(a) \in \lexp(V^*) - \lexp(A^*)$, where
$u\in V^*$, $a\in A^*$.  By (\ref{eq:b-a-finite}), we have
$\overline{\lexp(u)} = \overline{\lexp(u_i)}$ for some index $i$,
in which case $\lexp(u) + \lexp(a_2) = \lexp(u_i) + \lexp(a_1)$ for some $a_1,a_2 \in A$.
Thus, $\lexp(u) - \lexp(a) = \lexp(u_i) + \lexp(a_1) - \lexp(a) - \lexp(a_2)$, and so
$\lexp(u) - \lexp(a) \in  \bigcup  \lexp(u_i) + \lexp(A^*) -
\lexp(A^*)$.  Thus, the forward inclusion has been demonstrated.  Since
the reverse inclusion is obvious, 
 (\ref{eq:b-a-union}) follows, and so
\begin{equation}\label{eq:b-a-union-2}
(\lexp(V^*) - \lexp(A^*) ) \setminus \lexp(V^*)
=   \bigcup_{i=1}^j (\lexp(u_i) + \lexp(A^*)  - 
\lexp(A^*)) \setminus
 \lexp(V^*)
. \end{equation}

Let $R$ be a reversely well-ordered subset of $(\lexp(V^*)-\lexp(A^*)) \setminus \lexp(V^*)$. In order to demonstrate that $R$ has finite cardinality, we know by  (\ref{eq:b-a-union-2}) that
it suffices to show that
\begin{equation}\label{eq:b-a-big}
R \cap (\lexp(u) + \lexp(A^*) -
\lexp(A^*))  \setminus \lexp(V^*) \end{equation}
is a finite set for all $u \in V$.
This follows directly from Lemma \ref{lemma:finite-difference-prep}. \end{proof}

\begin{lemma}\label{lemma:abc}
Suppose $W$ is a $K$-vector space such that $V \subseteq W \subseteq L$.
 If $\lexp(V^*)/\lexp(A^*) = \lexp(W^*)/\lexp(A^*)$, then
$\lexp(W^*) \subseteq \lexp(V^*) - \lexp(A^*)$.
\end{lemma}

\begin{proof}
Let $\alpha \in \lexp(W^*)$. Since $\lexp(V^*)/\lexp(A^*) =
\lexp(W^*)/\lexp(A^*)$, we know that $\overline{\alpha} = \overline{\lexp(u)}$ for some $u\in V^*$. Thus, for some $a_1, a_2 \in A$,
we have $\alpha  + \lexp(a_1) = \lexp(u) + \lexp(a_2)$.  Therefore, $\alpha = \lexp(u) + \lexp(a_2) -
\lexp(a_1) \in \lexp(V^*) - \lexp(A^*)$.
\end{proof}

We are now in a position to produce a set of conditions that guarantees that the
bound produced in  Lemma  \ref{lemma:digging} is tight. 
First, we define
the $K$-vector space $A^{-1}V$ by
\begin{equation*}\label{local}
A^{-1}V = \{ v/a \mid a \in A^*, v\in V\}.
\end{equation*}

\begin{theorem}\label{theorem:quotientplusone}
Suppose the vector space extension $A \subseteq V$ is finite,
$v|_A$ is discrete of rank 1, and $\lexp(A^*)$ is nonpositive.
Given $c \in L \setminus A^{-1}V$ such that $\lexp((V+Ac)^*)$ is reversely well ordered, there is exactly one element in
$\lexp((V+Ac)^*)/\lexp(A^*)$ that is not in $\lexp(V^*)/\lexp(A^*)$.
\end{theorem}

\begin{proof}
Since $\lexp(A^*)$ is nonpositive, we can fix $a \in A$ such that $v(a) < 0$. Since $\lexp$ is trivial on $K$, for every polynomial $p(x)$ in the polynomial ring $K[x]$, we have $\lexp(p(a)) = (\deg_x f) \lexp(a)<0$,  and so $\lexp(p(a))$ obviously cannot be $\infty$. Therefore, $p(a) =0$ precisely when the coefficients of $p(x)$ are all zero, and so $a$ must be transcendental. Define
$$W = \{ u + p(a) c \mid u \in V, \ p(x) \in K[x]^*\}.$$
 The fact that  $\lexp((V+Ac)^*)$ is reversely well ordered guarantees that $\lexp(W^*)$ is reversely well ordered.

By Lemma \ref{lemma:digging}, there is at most one element in
$\lexp((V+Ac)^*)/\lexp(A^*)$ that is not in $\lexp(V^*)/\lexp(A^*)$.
Suppose, for contradiction, that
\begin{equation}
\lexp((V+Ac)^*)/\lexp(A^*) = \lexp(V^*)/\lexp(A^*).
\end{equation}
From the inclusions $V \subseteq W \subseteq V+Ac$, it follows that
\begin{equation}
\lexp(W^*)/\lexp(A^*) = \lexp(V^*)/\lexp(A^*).
\end{equation}

Since $\lexp(W^*)$ is reversely well ordered, the set $S$ given by
\begin{equation}\label{eq:defSimplicit}
S= \lexp(W^*) \cap (\lexp(V^*)-\lexp(A^*))  \setminus \lexp(V^*)
\end{equation}
must   be reversely well ordered, and hence by Lemma \ref{lemma:finite-difference} it must have finite cardinality.
  By Lemma \ref{lemma:abc},
$$\lexp(W^*) \subseteq \lexp(V^*) - \lexp(A^*),$$ and so
 $$S = \lexp(W^*) \setminus \lexp(V^*).$$

Write $S = \{s_1, \dots, s_n\}$, and for each index $j$, 
fix $w_j \in V$ and an element $p_j(x)$ in the polynomial ring $K[x]$ such that $\lexp(w_j+p_j(a)c) = s_j$. Define $\delta_j = \deg_x(p_j)$ and $\delta = \max\{\delta_1, \dots, \delta_n\}$.

We will define an infinite sequence $\{ f_i \}_{i \in \Z_{\ge 0}}$ of nonzero elements of $W$ of
the form $f_i = u_i + q_i(a) c$, where $u_i \in V$ and  $q_i$ is a polynomial of degree greater than $\delta$. Furthermore, our sequence will be constructed so that 
$\{ \lexp(f_i) \}_{i \in \Z_{\ge 0}}$ is increasing, which contradicts the assertion that $\lexp(W^*)$ is reversely well ordered.

Since $S$ has finite cardinality, we know that for $d \gg 0$,
\begin{equation}\label{eq:dgg0}
\lexp(w_1a^d + p_1(a)a^d c) = \lexp(w_1+p_1(a)c) + d \lexp(a) \not\in S.
\end{equation}
Select $d > \delta$ such that (\ref{eq:dgg0}) holds, in which case
$$
\lexp(w_1a^d + p_1(a)a^d c) \in \lexp(V^*).
$$
We define $f_0 = u_0 + q_0(a) c$, where $u_0= w_1a^d$ and $q_0(x) =  p_1(x)x^d$.

Given $f_0, \dots, f_{i-1}$, we show how to construct $f_i $. We divide this into two cases, depending on whether $\lexp(f_{i-1}) \in \lexp(V^*)$.

\begin{enumerate}
\item[Case 1:] Suppose $\lexp(f_{i-1}) \in \lexp(V^*)$, in which case
$\lexp(f_{i-1}) = \lexp(u_{i-1}')$ for some $u_{i-1}' \in V^*$. Then there exists $\mu_{i-1} \in K^*$ such that $\lexp( \mu_{i-1} u_{i-1}' + f_{i-1} ) > \lexp (f_{i-1})$.
Define $u_i = \mu_{i-1}u_{i-1}'+u_{i-1}$ and $q_i(x) = q_{i-1}(x)$. If we  define $f_i = u_i + q_i(a)c$, then 
$f_i = \mu_{i-1}u_{i-1}' + u_{i-1}  +  q_{i}(a) c = \mu_{i-1}u_{i-1}' + u_{i-1}  +  q_{i-1}(a) c = \mu_{i-1}u_{i-1}'   +  f_{i-1},$
and so $\lexp(f_i) > \lexp(f_{i-1})$. Moreover, the degree of $q_i(x)$ is greater than $\delta$.

\item[Case 2:] Suppose $\lexp(f_{i-1}) \not\in \lexp(V^*)$, in which case $\lexp(f_{i-1}) \in \lexp(W^*) \setminus \lexp(V^*) = S$, and so for 
 some index $j$, we have 
$$\lexp(f_{i-1}) = s_j = \lexp(w_j+p_j(a)c)$$
where $w_j \in V$ and  $\deg p_j(x) \le \delta < \deg q_{i-1}(x)$.  We now show that $f_{i-1}$ and $w_j+p_j(a)c$ are not $K$-scalar multiples of one another. Indeed, if $u_{i-1} + q_{i-1}(a)c = \lambda(w_j + p_j(a) c)$ for some $\lambda \in K$, then $(q_{i-1}(a) - \lambda p_j(a))c = \lambda w_j-u_{i-1}\in V$, which is impossible since $a$ is transcendental over $K$ and $c \not\in A^{-1}V$.
Therefore,  there exists $\lambda \in K$ such that $\lexp(u_{i-1} + q_{i-1}(a)c+ \lambda(w_j + p_j(a) c)) > \lexp
(u_{i-1} + q_{i-1}(a)c).$
If we define $f_i = u_i + q_i(a)c$ where
$u_i = u_{i-1} + \lambda w_j$ and $q_i(x) = q_{i-1}(x) + \lambda p_j(x)$, it follows that $\lexp(f_i) > \lexp(f_{i-1})$.
Moreover, 
since $\deg q_{i-1} > \delta \ge \deg p_j$ and $q_i = q_{i-1} + \lambda p_j$, it follows that $\deg q_i = \deg q_{i-1} >\delta$.
\end{enumerate}
\end{proof}

\section{Valuations on  Polynomial Rings in Two Variables}\label{valpoly}

We now take the results from Section \ref{section:growth} and apply them to polynomial rings in two variables. Throughout this section, $\lexp$ will be a $K$-valuation on $K(x,y)$ such that $\lexp(K[x,y]^*)$ is reversely well ordered.

\begin{definition}\label{Lambda}
For each $i \in \Z_{\ge 0}$, we define  
\begin{equation*}\label{def:Lambdai}
\V_i = \{ f \in K[x,y]  \mid \deg_yf \le i\}.
\end{equation*}
\end{definition}

In the sequel, for  any  $\alpha \in \lexp(K(x,y)^*)$,  define $\overline{\alpha}$ to be its  image  in $\lexp(K(x,y)^*)/\lexp(\V_0^*)$. Similarly, when $S \subseteq \lexp(K(x,y)^*)$,  define $\overline{S}$ to be its  image in $\lexp(K(x,y)^*)/\lexp(\V_0^*)$.

In general, when $\lexp$ is a $K$-valuation on $K(x,y)$, Lemma
\ref{prop:digging-card} can be used to show that the quotient
$\lexp(\V_n^*)/\lexp(\V_{0}^*)$ has cardinality at most $n+1$.  However, we will see in the next lemma that when
$\lexp(K[x,y]^*)$ is reversely well ordered, this bound is tight.

\begin{lemma}\label{lemma:Pgradual}
Define $f_0 = 1$. There exists $f_i  \in K[x,y]$ with
$\deg_y f_i= i$
for
each $i \in \Z_{\ge 0}$
such
that the following  conditions hold:
\begin{enumerate}
 \item[(i)] $\overline{\lexp(K[x,y]^*)}  = \{
\overline{\lexp(f_i)} \mid i \in \Z_{\ge 0} \}$;
 \item[(ii)] $\overline{\lexp(\V_\ell^*)}  = \{
\overline{\lexp(f_i)} \mid
0 \le i \le \ell \}$ for any $\ell \in \Z_{\ge 0}$;
 \item[(iii)] 
 for $i \not = j$, we have $\overline{\lexp(f_i)} \neq\overline{\lexp(f_j)}$.
\end{enumerate}
\end{lemma}

\begin{proof}
If we set $A=V_0=K[x], V= V_{i-1} =  \{ f \in K[x,y] \mid \deg_y f_i \le i-1\}$, $L=K(x,y)$ and $c = y^i$, then by Theorem
\ref{theorem:quotientplusone}, it follows that  $\overline{\lexp(\V_i^*)}
\setminus \overline{\lexp(\V_{i-1}^*)}$ is a singleton.
Therefore,
 $\overline{\lexp(\V_i^*)}
\setminus \overline{\lexp(\V_{i-1}^*)} = \{ \overline{\lexp(f_i)} \}$
for some
$f_i \in K[x,y]$ such that $\deg_y f_i = i$.
Thus,
\begin{equation}
\overline{\lexp(K[x,y]^*)} =  \bigcup_{i \in \Z_{\ge 0}}
 \overline{\lexp(\V_i^*)}  = \{ \overline{\lexp(f_i)} \mid i \in \Z_{\ge 0} \}
\end{equation}
and
\begin{equation}
\overline{\lexp(\V_\ell^*)} =   \{ \overline{\lexp(f_i)} \mid 0 \le i 
\le \ell \}.
\end{equation}
Since
 $\overline{\lexp(\V_i^*)}
\setminus \overline{\lexp(\V_{i-1}^*)}$ is a singleton set for each index $i$, we have that
$\overline{\lexp(f_i)} \neq\overline{\lexp(f_j)}$ whenever $i \neq j$.
\end{proof}

\begin{lemma}\label{lemma:gradual-alt}
Suppose $g_i \in K[x,y]^*$ such that
$\deg_y g_i = i$ for $i \in
\Z_{\ge 0}$.
If $\overline{\lexp(g_i)} \neq \overline{\lexp(g_j)}$ for all $i\neq j$, then
\begin{equation}
\overline{\lexp(K[x,y]^*)} = \{ \overline{\lexp(g_i)} \mid i
\in \Z_{\ge 0} \}.
\end{equation}
\end{lemma}

\begin{proof}
We know from Lemma \ref{lemma:Pgradual} that $\#
\overline{\lexp(\V_i^*)}
= i+1$, and since $\overline{\lexp(g_0)},  \dots,
\overline{\lexp(g_i)}$ are distinct elements of 
$\overline{\lexp(\V_i^*)}$,
it follows that
 $$\overline{\lexp(\V_i^*)} = \{ \overline{\lexp(g_0)},
\overline{\lexp(g_1)}, \dots,
\overline{\lexp(g_i)} \},$$ and so
$$\overline{\lexp(K[x,y]^*)} =
\overline{\bigcup_{i=0}^\infty \lexp(\V_i^*)}
=\bigcup_{i=0}^\infty \overline{\lexp(\V_i^*)}
=
\{ \overline{\lexp(g_i)} \mid i
\in \Z_{\ge 0} \}.$$
\end{proof}

\begin{definition}
We say that $\alpha\neq 0$ is an {\bf indivisible element}  of the group $\Gamma$ if there does not exist an integer $n \ge 2$ and $\gamma \in \Gamma$ such that $\alpha = n \gamma$.
We say that $\alpha, \beta \in \Gamma^*$ are {\bf commensurable} if there exist $m,n \in \Z^*$ such that $m \alpha = n \beta$.
\end{definition}

Going forward, we focus on the case when $\lexp(K(x,y)^*) = \Z \oplus \Z$.

\begin{prop}\label{prop:alphabetadrop}
Suppose  $\lexp(K(x,y)^*) = \Z \oplus \Z$, and let $\alpha$ be an indivisible element of $\Z \oplus \Z$ such that
$\lexp(x), \lexp(y) \in \Z_{\ge 0}\alpha$. If $\lexp(x) = m \alpha$, where $m \in \Z_{\ge 0}$, then the following statements  hold:
\begin{enumerate}

\item[(i)] $\lexp(\V_{m-1}^*) \subseteq \Z_{\ge 0} \alpha$;

\item[(ii)] $\lexp(\V_{m}^*) \not\subseteq \Z_{\ge 0}\alpha$.

\end{enumerate}

\end{prop}

\begin{proof}
Since $\lexp(x) = m\alpha$, it follows that $\lexp(\V_0^*) =  \Z_{\ge 0} m\alpha$. Moreover,  $\alpha \in \lexp(K(x,y)^*)$,
and so
\begin{equation}\label{Zalpha}
\overline{\Z \alpha}  = \{ \overline{i \alpha} \mid i \in \Z\} = 
\{ \overline{0}, \overline{\alpha}, \overline{2\alpha}, \dots, \overline{(m-1) \alpha}\}.
\end{equation}
By Lemma
\ref{lemma:Pgradual}, \begin{equation}
\overline{\lexp(K[x,y]^*)} = \{ \overline{\lexp(f_i)} \mid i \in \Z_{\ge 0} \},
\end{equation}
\begin{equation}\label{(m-1)reps}
\overline{\lexp(\V_{i-1}^*)}  = \{
\overline{\lexp(f_0)},
\overline{\lexp(f_1)}, \dots,
\overline{\lexp(f_{i-1})}
 \},
\end{equation}
and
\begin{equation}\label{diff:moduli}
\overline{\lexp(f_i)}
\neq
\overline{\lexp(f_j)} \ {\text{for $i \neq j$}},
\end{equation}
where  $f_i \in K[x,y]$ for each $i\in \Z_{\ge 0}$.
Given $\gamma_1, \gamma_2 \in \Z \oplus \Z$, we have that $\overline{\gamma_1}  =
\overline{\gamma_2}$ if and only if $\gamma_1 - \gamma_2 \in \Z m\alpha$.

We now demonstrate (i).
Let $\ell$ be the smallest integer such that  $f_\ell  \in K[x,y]$ 
with
$\deg_y f_\ell = \ell$
 where $\lexp(f_\ell) \not \in \Z_{\ge 0}\alpha $.  
Defining $\beta = \lexp(f_\ell)$,   it is clear that $\alpha$ and $\beta$ are not commensurable since $\alpha$ is indivisible.

Suppose, for contradiction, that $\ell < m$.
For each $i$, we have $f_i \in  K[x,y]$  with
$\deg_y f_i = i$,
and so the minimality of $\ell$ guarantees that $\lexp(f_i) \in \Z_{\ge 0}\alpha$ for $0 \le i
\le \ell -1$, in which case  we can write
\begin{equation}
\lexp(f_i) = (q_i m + r_i) \alpha,
\end{equation}
where $q_i, r_i \in \Z_{\ge 0}$ such that $0 \le r_i \le m-1$.  For $i, j \in \Z_{\ge 0}$ 
with $0 \le j \le \ell -1$, define
$g_{i \ell  + j} = f_\ell^i f_j$.  Note that $\deg_y (g_{i \ell + j}) = i
\deg_y(f_\ell) + \deg_y(f_j) = i \ell + j$, and so
\begin{equation}\label{eq:degy_prep}
\deg_y(g_s) = s
\end{equation}
for all $s \in \Z_{\ge 0}$.

Moreover,
\begin{equation}\label{eq:gij}
\lexp(g_{i\ell + j}) = i \lexp(f_\ell) + \lexp(f_j) = i \beta + q_jm \alpha + r_j
\alpha,
\end{equation} and so
$\overline{\lexp(g_{i \ell + j})} = \overline{i \beta + r_j \alpha}$.

Suppose for some indices $i_1, i_2, j_1, j_2$ with $0 \le j_1, j_2 \le \ell-1$, we have
$\overline{\lexp(g_{i_1\ell+j_1})} = \overline{\lexp(g_{i_2\ell+j_2})}$, in which
case $\overline{i_1 \beta + r_{j_1}\alpha } =\overline{i_2 \beta + r_{j_2}\alpha
}$.  Therefore, $i_1 \beta + r_{j_1} \alpha = i_2 \beta + r_{j_2} \alpha + zm
\alpha$ for some $z \in \Z$. Thus, $(i_1 - i_2) \beta = (r_{j_2} - r_{j_1} + zm)
\alpha$, and since $\alpha$ and $\beta$ are not commensurable,  $i_1 = i_2$. Moreover,
since $\overline{\lexp(g_{i_1\ell +j_1})}
= \overline{\lexp(g_{i_2\ell +j_2})}$,
it follows from (\ref{eq:gij}) that
$\overline{i_1\lexp(f_\ell)+ f(j_1)} = 
\overline{i_2\lexp(f_\ell)+ f(j_2)}$, and since
$i_1=i_2$, we have
$\overline{\lexp(f_{j_1})} = \overline{\lexp(f_{j_2})}$.
Thus by (\ref{diff:moduli}), 
we can conclude $j_1=j_2$, and so 
$i_1 \ell + j_1 = i_2 \ell + j_2$. 
Thus, we have shown
\begin{equation}\label{eq:gs_gt}
\overline{\lexp(g_s)} \neq \overline{\lexp(g_t)}
\end{equation}
whenever $s\neq t$.
Using this in conjunction with (\ref{eq:degy_prep}), we have by
Lemma \ref{lemma:gradual-alt} that \begin{equation}\label{eq:gs_dom}
\overline{\lexp(K[x,y]^*)} = \{ \overline{\lexp(g_i)} \mid i \in \Z_{\ge 0} \}.
\end{equation}

Since $\ell < m$, there exists $\lambda \in \{ 0, \dots, m-1 \} \setminus \{r_0,
\dots, r_{\ell-1}\}$.  Write $\lambda = qm+r$ where $0 \le r \le m-1$. Since
$\lambda \alpha  \in \Z \oplus \Z= \lexp(K(x,y)^*)$, there exist $h_1, h_2 \in K[x,y]$ such that
$\lambda \alpha = \lexp(h_1/h_2) = \lexp(h_1) - \lexp(h_2)$.
By (\ref{eq:gs_dom}), there exist indices $i_1,j_1, i_2, j_2$ 
with $0 \le j_1, j_2 \le \ell-1$
such that
$\overline{\lexp(h_1)} = \overline{\lexp(g_{i_1 \ell + j_1})}$ and
 $\overline{\lexp(h_2)} = \overline{\lexp(g_{i_2 \ell + j_2})}$, in which case $\overline{\lexp(h_1)} = \overline{i_1 \beta + r_{j_1}\alpha}$ and $\overline{\lexp(h_2)} = \overline{i_2 \beta + r_{j_2} \alpha}$.
Thus, $\lexp(h_1) = i_1 \beta + r_{j_1} \alpha + s_1 m \alpha$ and
$\lexp(h_2) = i_2 \beta + r_{j_2} \alpha + s_2 m \alpha$ for some $s_1, s_2 \in \Z$.
Thus,
\begin{equation}
\lambda \alpha = (i_1 - i_2) \beta + (r_{j_1} - r_{j_2}) \alpha + (s_1 - s_2) m \alpha.
\end{equation}
Since $\alpha$ and $\beta$ are not commensurable, it follows that
$i_1 = i_2$ and so
\begin{equation}
\overline{\lambda \alpha} = \overline{ (r_{j_1} - r_{j_2})
\alpha }.
\end{equation}
Now, whenever $i=0$ and $0 \le j \le \ell-1$, we have by (\ref{eq:gij}) that 
$\overline{\lexp(g_{i\ell+j})} = \overline{r_j \alpha}$, and so
 \begin{eqnarray*}
\overline{\lexp(g_{j_1}^{} g_{j_2}^{m-1})} &= &  \overline{ \lexp(g_{j_1}) +
(m-1) \lexp(g_{j_2})}\\
& = &  \overline{r_{j_1}\alpha + (m-1) r_{j_2} \alpha} \\
& = & \overline{m  r_{j_2} \alpha + (r_{j_1} - r_{j_2}) \alpha }\\
& = & \overline{(r_{j_1} - r_{j_2}) \alpha }.
 \end{eqnarray*}
Therefore, $\overline{\lambda \alpha} = \overline{\lexp(g_{j_1}^{}
g_{j_2}^{m-1})}$, and so by (\ref{eq:gs_dom}), we have $\overline{\lambda \alpha} = \overline{\lexp(g_k)}$
for some $k \in \Z_{\ge 0}$.
Now, by (\ref{eq:gij}), we know that $\lexp(g_s) \in \Z_{\ge 0}\alpha$ if and only if
$s \le \ell -1$.
Thus, $k \le \ell -1$, and so $\overline{\lambda \alpha} = \overline{\lexp(g_k)} =
\overline{r_k \alpha}$.  Since $0 \le \lambda \le m-1$ and $0 \le r_k \le m-1$, it
follows from the fact that $\lexp(\V_0^*) =  \Z_{\ge 0} m\alpha$ that $\lambda = r_k$, which contradicts the assumption that
$\lambda \in \{ 0, \dots, m-1\} \setminus \{r_0, \dots, r_{\ell -1}\}$.
Therefore, $\ell \ge m$, and so part (i) follows.

Given that $\lexp(\V_{m-1}^*) \subset \Z_{\ge 0} \alpha$, 
it follows from (\ref{(m-1)reps}) and (\ref{diff:moduli})
together that $\overline{\lexp(\V_{m-1}^*)}$ has cardinality $m$, and so
\begin{equation}\label{level(m-1)}
 \overline{\lexp(\V_{m-1}^*)} = \{ \overline{0}, \overline{\alpha}, \overline{2\alpha}, \dots, \overline{(m-1) \alpha}\}.
\end{equation}
Suppose, for contradiction, 
$\lexp(\V_m^*) \subset \Z_{\ge 0}\alpha$. Therefore, $\lexp(f_m) \in \Z_{\ge 0} \alpha$, in which case by (\ref{Zalpha}), $\overline {\lexp(f_m)} = \overline{i \alpha}$ for some $0 \le i \le m-1$. However, we have just seen by (\ref{level(m-1)}) that $\overline{i \alpha} = \overline{\lexp(f_j)}$ for some index $0 \le j \le m-1$, and so $\overline{\lexp(f_m)} = \overline{\lexp(f_j)}$, which contradicts  
(\ref{diff:moduli}).
\end{proof}

Given two polynomials $f,w \in K(x)[y]$, we will often need to write $f$ as  an expansion in $w$, which is  accomplished by taking $f$ and iteratively dividing by $w$ to obtain a sequence of quotients and remainders. This expansion is made precise by the following result, which we state without proof.

\begin{lemma}\label{decompose}
Let $w \in K(x)[y]$, and define $m = \deg_y w$. Given $f \in K(x)[y]$, there is an index $\ell$
and unique  $f_{i,j} \in K(x)$
such that  
\begin{equation}\label{w-expansion}
f= \sum_{i=0}^\ell \sum_{j=0}^{m-1} f_{i,j} y^j w^i.
\end{equation}
We call (\ref{w-expansion}) the $w$-{\bf expansion of} $f$.
Note that if $\ell$ is larger than necessary, the additional coefficients  $f_{i,j}$ are all zero. We say ``the expansion" instead of ``an expansion" even though the expansion is unique only up to the choice of $\ell$ and the addition of  coefficients that are just ``0". We call  each $f_{i,j}y^jw^i$ a {\bf term of the $w$-expansion}.
\end{lemma}

When the value group is $\Z \oplus \Z$, we can use these expansions to prove a result about $\lexp(K[x,y]^*)$ under 
the assumption it is reversely well ordered. Specifically, we have the following proposition.

\begin{prop}\label{prop:Zalphabeta}
Suppose  $\lexp(K(x,y)^*) = \Z \oplus \Z$ and that $\lexp(K[x,y]^*)$ is reversely well ordered. Suppose further that $\alpha$ is an indivisible element of $\Z \oplus \Z$ such that
$\lexp(x) = m\alpha$ and $\lexp(y)  =  n\alpha$, where $m,n\in \Z_{\ge 0}$.
If  $w \in K[x,y]^*$ such
that $\deg_y w= m$
where $\beta = \lexp(w)$ and $\alpha$ are not commensurable,
then
$$\lexp((K(x)[y])^*) \subset \alpha\Z  + \beta\Z_{\ge 0} .$$
\end{prop}

\begin{proof}
Given  $f(x,y) \in K(x)[y]^*$,  consider the $w$-expansion
\begin{equation}
f= \sum_{i=0}^\ell \sum_{j=0}^{m-1} f_{i,j} y^j w^i.
\end{equation}
Simultaneously clear all the denominators by choosing $h \in K[x]^*$ so that
$hf_{i,j} \in K[x]$.
For each index $i$, since $0 \le j \le m-1$, we have by  Proposition~\ref{prop:alphabetadrop} that
$$ \lexp\left(\sum_{j=0}^{m-1} hf_{i,j} y^j\right)  = \lambda_i \alpha,$$
for some $\lambda_i \in \Z_{\ge 0}$.

  If
$\lexp\left(\sum_{j=0}^{m-1} hf_{i_1,j} y^jw^{i_1}\right) = \lexp\left(\sum_{j=0}^{m-1} hf_{i_2,j} y^jw^{i_2}\right)$ for some   pair of indices $i_1,i_2$,  we have $\lambda_{i_1} \alpha + {i_1} \beta = \lambda_{i_2}
\alpha + {i_2} \beta$, and so
\begin{equation}
 (\lambda_{i_1}  - \lambda_{i_2}) \alpha = (i_2-i_1) \beta,
\end{equation}
in which case $i_1=i_2$ since $\alpha$ and $\beta$ are not commensurable.

Therefore, whenever $i_1 \neq i_2$, we have 
$\lexp(\sum_{j=0}^{m-1} hf_{i_1,j} y^jw^{i_1}) \neq \lexp(\sum_{j=0}^{m-1} hf_{i_2,j} y^jw^{i_2})$, and so by the strong triangle inequality,
\begin{equation}
 \lexp( f ) = \min_{0 \le i \le \ell} \lexp\left(\sum_{j=0}^{m-1} hf_{i,j} y^jw^{i}\right) - \lexp(h) =  \min_{0 \le i \le \ell}
(\lambda_i \alpha + i \beta) -\lexp(h).
\end{equation}
The conclusion follows from the observation that $\lexp(h) \in \alpha\Z_{\ge 0} $.
\end{proof}




\section{Constructing Explicit Valuations}\label{section:construction}

Using the terminology introduced in Lemma~\ref{decompose}, we construct a special class of $K$-valuations on $K(x,y)$.
We begin by stating the following lemma without proof.

\begin{definition}\label{valonx}
The degree valuation $v_\infty: K(x) \to \Z \cup \{ \infty \}$ is a $K$-valuation on $K(x)$
given by $v_{\infty}(0) = \infty$ and
\begin{equation}
v_{\infty}\left(f/g \right) =  \deg_x(g) - \deg_x(f)
\end{equation}
for $f,g \in K[x]^*$.
\end{definition}

From this definition, the simple corollary below follows by considering polynomial division.

\begin{cor}\label{polyrat}
 $K$ is a field of representatives for the residue field
${\mathcal O}_{v_\infty}/{\mathcal M}_{v_\infty}$.
Moreover, given $g \in K(x)^*$, there exists a unique $f \in K[x]$ such that $v_{\infty}(f+g) >  0$.
\end{cor}

We now define a set of maps on $K(x,y)$ that serve as candidates for $K$-valuations. After constructing such maps in terms of multiple parameters, we prove some intermediate results that lead to Proposition \ref{homomorphism:powers}, which demonstrates that these maps are, indeed, $K$-valuations when the given parameters satisfy specific conditions.

\begin{definition}
\label{def:constructval}
Let $m, n$ be positive, relatively prime integers,
and let $w \in K(x)[y]$ be monic in  $y$ with  $\deg_y w = m$.
Let $\alpha, \beta$ be nonzero, indivisible elements of $\Z \oplus \Z$ such that are not commensurable.   We define  {\bf the map associated to $(m,n, w, \alpha, \beta)$} as
 $$\lexp: K(x)[y] \to \Z \oplus \Z \cup \{ \infty\}$$  by setting $\lexp(0) = \infty$, and
for $f \in K(x)[y]^*$, defining
 \begin{equation}\label{prop:lexpconstruction1}
\lexp(f) =  \min_{0 \le i \le \ell \atop 0 \le j \le m-1} \left\{ -v_{\infty} (f_{ij})m \alpha + jn \alpha +  i  \beta \right\}, 
\end{equation}
where $f$ has the $w$-expansion
$$f= \sum_{i=0}^\ell \sum_{j=0}^{m-1} f_{i,j} y^jw^i.$$
\end{definition}

Throughout the sequel, we will assume $\alpha$ to be negative so that it is appropriate for use in the definition above.

\begin{lemma}\label{uniqueterm}
Let $\lexp$ be the   map  associated to 
$(m,n, w, \alpha, \beta)$, and let $f,g \in K(x)[y]^*$.
Suppose for some $i_1, i_2, j_1, j_2 \in \Z_{\ge 0}$ such that $0 \le j_1, j_2 \le m-1$,
\begin{eqnarray*}
\lexp(f) & =  & -v_{\infty}(f_{i_1,j_1})m \alpha + {j_1}n \alpha +i_1  \beta;\\
\lexp(g) &  = & -v_{\infty}(g_{i_2,j_2})m \alpha + {j_2} n\alpha + i_2  \beta,
\end{eqnarray*}
where $f_{i_1,j_1}, g_{i_2,j_2}$ are terms of the $w$-expansions of $f, g$, respectively.
If $\lexp(f) = \lexp(g)$, then $i_1 = i_2$, $j_1=j_2$, and $v_\infty(f_{i_1,j_1}) = v_{\infty}(g_{i_2,j_2})$.
\end{lemma}

\begin{proof}
Suppose $\lexp(f) = \lexp(g)$. Note that $f_{i_1,j_1}, g_{i_2,j_2} \neq 0$; otherwise, $\lexp(f)$ and $\lexp(g)$ could not be written in the form above.
Then $$ (i_1-i_2) \beta = v_{\infty}(f_{i_1,j_1}/g_{i_2,j_2}) m \alpha+  ({j_2} -{j_1})n \alpha.$$
Since $\alpha$ and $\beta$ are not commensurable, $i_1 = i_2$. Moreover, since $m$ and $n$ are relatively prime, it follows that 
 ${j_2}-{j_1}$ is a multiple of $m$.
However, 
  $0 \le j_1, j_2 \le m-1$,
and so $j_1 = j_2$.
 Thus, $v_{\infty} (f_{i_1,j_1}) = v_{\infty}(g_{i_2,j_2})$.
\end{proof}

Setting $f=g$ in the lemma above,
we conclude that there is a unique term in the $w$-expansion
of $f$ that gives rise to its image under $\lexp$.  This leads us to the following definition.

\begin{definition}\label{leadterm}
The term $f_{i,j}y^jw^i$ in the $w$-expansion of $f$
such that $\lexp(f) = \lexp(f_{i,j}y^jw^i)$ is called the {\bf lead term} of $f$ with respect to $(m,n, w, \alpha, \beta)$.  We denote this lead term by $\tau(f)$.
\end{definition}

\begin{lemma}\label{constructvallemma}
The map $\lexp$ 
 associated to 
$(m,n, w, \alpha, \beta)$ has the following properties for all $f,g \in K[x,y]$:
\begin{enumerate}
\item $\lexp(f+g) \ge \min\{\lexp(f), \lexp(g)\}$;
   \item $\lexp(f+g) = \min\{\lexp(f), \lexp(g)\}$ whenever $\lexp(f) \neq \lexp(g)$;
\item If $\lexp(f) = \lexp(g) \neq \infty$, then
$\exists!\lambda\in {K}$ such that  $\lexp(f+\lambda g) > \lexp(f)$.
\end{enumerate}
\end{lemma}

\begin{proof}
We write the $w$-decompositions of $f$ and $g$ using an index $\ell$ that is large enough to accommodate both expressions:
\begin{eqnarray*}
f&=& \sum_{i=0}^{\ell} \sum_{j=0}^{m-1} f_{i,j} y^jw^i;\\
g&=& \sum_{i=0}^{\ell} \sum_{j=0}^{m-1} g_{i,j} y^jw^i.
\end{eqnarray*}
Consequently, for an appropriate selection of indices, we have the following:
\begin{eqnarray*}
\tau(f) &= &f_{i_1,j_1} y^{j_1} w^{i_1}\\
\tau(g) &= &g_{i_2,j_2} y^{j_2} w^{i_2} \\
\tau(f+g) &= & (f_{i_3,j_3} + g_{i_3,j_3}) y^{j_3} w^{i_3} \\
\end{eqnarray*}
Since $\alpha$ is negative, the inequality $v_{\infty}(f_{i_3,j_3} + g_{i_3,j_3}) \ge \min \{ v_{\infty}(f_{i_3,j_3} ), v_{\infty}(g_{i_3,j_3}) \}$
implies 
\begin{eqnarray*}
\lexp(f+g) 
& = & -v_{\infty}(f_{i_3,j_3} + g_{i_3,j_3})m \alpha + {j_3}n \alpha + i_3  \beta   \\
& \ge &  \min \{ -v_{\infty}(f_{i_3,j_3} )m \alpha + {j_3}n \alpha +i_3  \beta , -v_{\infty}(g_{i_3,j_3})m \alpha + {j_3}n \alpha + i_3  \beta  \}\\
& \ge  & \min \{ \lexp(f), \lexp(g)\},
\end{eqnarray*}
thus justifying property (i).

If $\lexp(f) \neq \lexp(g)$, then without loss of generality we assume that $\lexp(f) <\lexp(g)$. Thus, $\lexp(f) = \lexp(\tau(f))$ is less than the $\lexp$-image of each term appearing in the $w$-expansion of $g$. By definition, $\lexp(\tau(f))$ is less than the $\lexp$-image of any other term appearing in the $w$-expansion of $f$. Given arbitrary terms $T_f, T_g$ in $f, g$, respectively, we can apply property (i) to obtain $\lexp(T_f + T_g) \ge \min \{ \lexp(T_f), \lexp(T_g) \} \ge \lexp (\tau(f))$. Therefore, $\lexp(\tau(f))$ is less than or equal to the $\lexp$-image of any term appearing the in the $w$-expansion of $f+g$. 

If $g$ does not have a term of the form $g_{i_1,j_1}y^{j_1}w^{i_1}$,
 then $\tau(f)=f_{i_1,j_1}y^{j_1}w^{i_1}$ is a term of $f+g$ and so 
 $$\lexp(f+g) =  \lexp(\tau(f)) = \lexp(f)=  \min\{\lexp(f), \lexp(g)\}.$$ 
 This leaves us with the possibility where $g$ does have a term of the form $g_{i_1,j_1}y^{j_1}w^{i_1}$, in which case $(f_{i_1,j_1}+g_{i_1,j_1}) y^{j_1} w^{i_1}$ is a term
 of $f+g$. We previously stated that $\lexp(\tau(f))$ is less than the $\lexp$-image of any other term appearing in the expansion of $g$, and so 
 $\lexp(f_{i_1,j_1}) < \lexp(g_{i_1,j_1})$.
 Thus,
 $\lexp((f_{i_1,j_1}+g_{i_1,j_1}) y^{j_1} w^{i_1})=
 \lexp(f_{i_1,j_1} y^{j_1} w^{i_1})= \lexp(f)$, 
 and so 
 $$\lexp(f+g) = \lexp(f)=  \min\{\lexp(f), \lexp(g)\}.$$
 Thus, we have justified property (ii).

To justify property (iii), we begin by assuming $\lexp(f) = \lexp(g) \neq \infty$. By property (i), we have that $\lexp(f+ \lambda g) \ge \lexp(f)$, and so we only need to show that $\lexp(f+ \lambda g) \neq \lexp(f)$.
By Lemma \ref{uniqueterm},  we have $i_1=i_2$, $j_1=j_2$, and 
$v_{\infty}(f_{i_1,j_1}) = v_{\infty}(g_{i_2,j_2})$. Thus, there exists $\lambda \in K$ such that $v_{\infty}(f_{i_1,j_1} + \lambda g_{i_2,j_2}) > v_{\infty}(f_{i_1,j_1})$.
Write $\tau(f+\lambda g) = (f_{i_5,j_5}  + \lambda g_{i_5,j_5})y^{j_5}w^{i_5}$ for an appropriate choice of indices.
If $\lexp(f + \lambda g) = \lexp(f)$, then by Lemma \ref{uniqueterm}, we have that $i_1 = i_5$ and $j_1 = j_5$ and 
$v_{\infty}(f_{i_5,j_5} + \lambda g_{i_5,j_5}) = v_{\infty}(f_{i_1,j_1})$, in which case
$v_{\infty}(f_{i_1,j_1} + \lambda g_{i_1,j_1}) = v_{\infty}(f_{i_1,j_1})$. Since $i_1=i_2$ and $j_1=j_2$, we have $v_{\infty}(f_{i_1,j_1} + \lambda g_{i_2,j_2}) = v_{\infty}(f_{i_1,j_1})$, a contradiction.

The uniqueness of $\lambda$ follows from the observation that the residue field of $v_{\infty}$ is $K$.
 \end{proof}

The proof of the lemma below follows from the observation that $$\tau(f(x)g(x,y)) = f(x) \tau(g(x,y)) = \tau(f(x)) \tau(g(x,y))$$ when $f(x) \in K(x)$ and $g(x) \in K(x)[y]$.

\begin{lemma}\label{homomorphism:multiplyx}
If $\lexp$ is the map
 associated to 
$(m,n, w, \alpha, \beta)$, then for all
 $f(x) \in K(x), g(x,y) \in K(x)[y]$,
$$\lexp(f(x) g(x,y)) = \lexp(f(x)) + \lexp(g(x,y)).$$
\end{lemma}

We see in Definition \ref{def:constructval} that the map associated to $(m,n, w, \alpha, \beta)$  satisfies property (i) of Definition \ref{def:valuation}. Moreover, 
we demonstrated in Lemma \ref{constructvallemma} that this map also satisfies properties (iii) and (iv) of Definition \ref{def:valuation}.  If we wish to show that this map also satisfies property (ii) of  
Definition \ref{def:valuation}, then we must require additional conditions. The result below provides a set of conditions that guarantees that the map will be a $K$-valuation.

\begin{prop}\label{homomorphism:powers}
Let $\lexp$ be the map associated  
to $(m,n, w, \alpha, \beta)$ where 
$w = y^m +  \sum_{k=0}^{m-1} w_ky^k$
with $w_k \in K(x)$. 
Suppose  further we have
\begin{enumerate}
\item[(i)] $\beta > mn \alpha$;
\item[(ii)] $\lexp (w_k) > (m-k)n \alpha$ for all $1 \le k \le m-1$;
\item[(iii)] $\lexp (w_0) = mn \alpha$.
\end{enumerate}
Then $\lexp$ is a $K$-valuation on $K(x,y)$.
\end{prop}

\begin{proof}
Due to Lemma \ref{constructvallemma}, it only remains to  show that for
all $f, g \in K[x,y]$, we have 
$\lexp(fg) = \lexp(f) +\lexp(g).$ To this end, we first prove by  induction on $j$ that 
\begin{equation}\label{eq:lexpyjwi}
\lexp(y^{j}w^i) = j n \alpha + i \beta
\end{equation} 
for all $i \in \Z_{\ge 0}$ and $0 \le j \le 2m-1$.
Since (\ref{eq:lexpyjwi}) follows from Definition \ref{def:constructval} for $0 \le j \le m-1$,
we need only justify the result when $m \le j \le 2m-1$.

We will assume that (\ref{eq:lexpyjwi}) holds for $0, 1, \dots, j-1$, and prove the result for the index $j$.
Write $j = m+r$ where $0 \le r \le m-1$, in which case
$$
y^j w^i = y^ry^m w^i  =y^r \left(w - \sum_{k=0}^{m-1} w_ky^k\right)  w^i,$$
and so
\begin{equation}\label{eq:ykexpansion}
y^j w^i = y^rw^{i +1}   - \sum_{k=0}^{m-1} w_ky^{k+r} w^i.
\end{equation}
We now show 
$\lexp(y^j w^i) = jn \alpha + i \beta$ by examining each term on the right-hand side of
(\ref{eq:ykexpansion}).

By  Definition \ref{def:constructval} and  assumption (i),
$$\lexp (y^r w^{i+1}) 
= rn \alpha  + (i +1) \beta 
= rn \alpha  + \beta + i \beta 
> (m+r)n \alpha + i \beta ,$$ 
and since $j=m+r$,
\begin{equation}\label{prelim1}
\lexp (y^r w^{i+1}) > jn \alpha + i \beta.
\end{equation}
For any $k< m$, we have $k+r< m+r=j$, and so  the strong induction hypothesis yields
\begin{equation*}\lexp (y^{k+r}  w^i) =  (k+r)n  \alpha + i \beta.
\end{equation*}
By Lemma \ref{homomorphism:multiplyx}, we have 
\begin{equation}\label{eq:ykr}\lexp (w_ky^{k+r}  w^i) = \lexp( w_k )+ (k+r)n  \alpha + i \beta.
\end{equation}
By assumption (ii), $\lexp (w_k) > (m-k)n \alpha$ for all $1 \le k \le m-1$,
and so
\begin{equation}\label{prelim2}
\lexp (w_ky^{k+r}   w^i)> (m-k)n \alpha + (k+r)n  \alpha + i \beta =(m+r)n \alpha + i \beta =  jn \alpha + i \beta
\end{equation}
for all $1 \le k \le m-1$.
Replacing $k$ by $0$ in (\ref{eq:ykr}) and applying assumption (iii), we obtain \begin{equation}\label{prelim3}
\lexp (w_0 y^r  w^i)= mn \alpha + rn \alpha + i \beta= jn \alpha + i \beta.
\end{equation}
Therefore, by (\ref{prelim1}), (\ref{prelim2}), and (\ref{prelim3}), it follows that $w_0 y^rw^i$ is the unique term in the expansion (\ref{eq:ykexpansion}) for which 
$\lexp$ achieves the minimum value $jn \alpha + i \beta$.

Thus, by part (ii) of Lemma \ref{constructvallemma},  
 the identity (\ref{eq:lexpyjwi}) holds 
for all $i \in \Z_{\ge 0}$ and $0 \le j \le 2m-1$.
Therefore, by Lemma \ref{homomorphism:multiplyx}, for all $h \in K(x)$, $i \in \Z_{\ge 0}$ and $0 \le j \le 2m-1$, 
\begin{equation}\label{hyw}
\lexp(hy^{j}w^i) =-  v_{\infty}(h) m \alpha + j n \alpha+ i \beta.
\end{equation}
 
Using this, we can complete our goal of demonstrating that 
$\lexp(fg) = \lexp(f) +\lexp(g)$ for all $f, g \in K[x,y]$.
For any pairs of terms
$f_{i_1,j_1} y^{j_1} w^{i_1}$ and
$g_{i_2,j_2} y^{j_2} w^{i_2}$ appearing in the $w$-expansions  of $f$ and $g$, respectively,
we have by  (\ref{hyw})  that
$$ \lexp(f_{i_1,j_1} y^{j_1} w^{i_1}g_{i_2,j_2} y^{j_2} w^{i_2}) =
 \lexp(f_{i_1,j_1} y^{j_1} w^{i_1}) + \lexp( g_{i_2,j_2} y^{j_2} w^{i_2}  ) .$$
Now, $fg$ is a linear combination of products of terms  coming from the 
 $w$-expansions  of $f$ and $g$, respectively.
 Let $f'$ and $g'$ be arbitrary terms of the $w$-expansions of $f$ and $g$, respectively, at least one of which is distinct from $\tau(f)$ and $\tau(g)$, respectively. Thus,  $\lexp(\tau(f)) \le \lexp(f')$ and $\lexp(\tau(g)) \le \lexp(g')$, with at least one of these inequalities being strict.
It follows that $\lexp(\tau(f)\tau(g)) = \lexp (\tau(f)) + \lexp(\tau(g)) < \lexp(f') + \lexp(g')
= \lexp(f'g')$, and so  of all the products of pairs of terms, $\tau(f)\tau(g)$ has the smallest image under $\lexp$. 
Consequently, 
$$\lexp(fg) = \lexp(\tau(f)\tau(g)) = \lexp(\tau(f)) + \lexp(\tau(g)) = \lexp(f) + \lexp(g).$$
\end{proof}

\section{The Examples}\label{examples}

In \cite{ms1}, an infinite family of valuations on the rational function field $K(x_1, \dots, x_n)$ was constructed such that  $\lexp(K[x_1, \dots, x_n]^*)$ is reversely well ordered.  These valuations are of the form $$\lexp:K(x_1, \dots, x_n) \to \Z^n$$ such that  for any polynomial $f \in K[x,_1, \dots, x_n]$, 
$$\lexp(f) = - \exp (\lm(\varphi(f))),$$
where $\exp$ represents the standard exponent vector of a monomial, $\lm(h)$ represents the leading monomial of the polynomial $h$ with respect to a fixed monomial order, and $\varphi$ is a $K$-algebra automorphism of the polynomial ring $K[x_1, \dots, x_n]$.  We say that such valuations {\bf come from a monomial order in suitable variables}. The following proposition from \cite{ms1} 
classifies all valuations on $K(x_1, \dots, x_n)$ that come from a monomial order in suitable variables.
\begin{prop}\label{suitvars}
If   $\lexp: K(x_1, \dots, x_n) \to \Z^n \cup \{\infty\}$ is a $K$-valuation, then $\lexp$ comes from a monomial order in suitable variables precisely when  $\lexp(K[x_1, \dots, x_n]^*)$ is isomorphic to the monoid $(\Z_{\le 0})^n$.

\end{prop}

With the aid of Proposition \ref{homomorphism:powers}, in Example  \ref{ex:suitablevals} we construct  valuations of the form $\lexp: K(x,y) \to \Z \oplus \Z$ such that $\lexp$ does not come from a monomial order in suitable variables, and yet $\lexp(K[x,y]^*)$ is still reversely well ordered.  
In the example below, we endow the value group  $\Z \oplus \Z$
with the lexicographic order, where the positive elements are ordered pairs $(a,b)$ such that either (i) $a>0$ or (ii) $a=0$ and $b > 0$.

\begin{example}\label{ex:suitablevals}
Let $\lexp$ be the map associated  
to $(m,n, w, \alpha, \beta)$ where 
$w = y^m +  x^n$, $\alpha = (-1,-1)$, $\beta = (0,-1)$, and $\Z \oplus \Z$ is endowed with the lexicographic order.
According to Proposition \ref{homomorphism:powers}, 
 $\lexp$ is a $K$-valuation on $K(x,y)$ because of the following observations:
 \begin{enumerate}
\item[(i)] $\beta$ has been selected so that $\beta > mn \alpha$;
\item[(ii)] $w_k=0$ whenever $1 \le k \le m-1$, in which case $\lexp (w_k)  = \infty$;
\item[(iii)] $\lexp (w_0) = \lexp(x^n) = mn \alpha$.
\end{enumerate}
Since $w \in K[x,y]$ is a monic polynomial in the variable $y$, the  $w$-expansion of each $f \in K[x,y]$ is of the form
$$f= \sum_{i=0}^\ell \sum_{j=0}^{m-1} f_{i,j} y^jw^i,$$
where each coefficient $f_{ij}$ is an element of $K[x]$. Thus, $v_\infty(f_{ij}) \le 0$
for each nonzero coefficient $f_{ij}$.
We have by (\ref{prop:lexpconstruction1}) that
 \begin{equation*}
 \lexp(f) =  \min_{0 \le i \le \ell \atop 0 \le j \le m-1} \left\{ -v_{\infty} (f_{ij})m \alpha + jn \alpha +  i  \beta \right\},
\end{equation*}
and so $$\lexp(f) \in \Z_{\ge 0} \alpha + \Z_{\ge 0} \beta \subset \Z_{\le 0} \oplus \Z_{\le 0},$$
from which it follows that $\lexp(K[x,y]^*)$ is reversely well ordered. 
However, since $(-1,0) \not \in \Z_{\ge 0} \alpha + \Z_{\ge 0} \beta$, it follows that 
$\lexp(K[x,y]^*) \neq \Z_{\le 0} \oplus \Z_{\le 0}$, and so by Proposition \ref{suitvars},
we see that $\lexp$ does not come from a monomial order in suitable variables.
\end{example}

In contrast to Example \ref{ex:suitablevals},  we construct a $K$-valuation $\lexp$ on $K(x,y)$ in  Example \ref{ex:not principal} such that $\lexp(K[x,y]^*)$ is nonpositive but not reversely well ordered. We will first do this by demonstrating that for this example, $\lexp(K[x,y]^*) \subset (\Z_{< 0} \oplus \Z) \cup \{(0,0)\} \subset \Z \oplus \Z,$ where $\Z \oplus \Z$ is endowed with the  lexicographic order.
Specifically, we identify a sequence $\{ f_i \}_{i=0}^\infty \subset K[x,y]$ such that $\deg_y f_i = 2(i+1)$ and $\lexp(f_i) = (-1,i-1)$.

When examining (\ref{prop:lexpconstruction1}),
we see that $\lexp(x^i) = i m\alpha$ for $i \in \Z_{\ge 0}$. However, a much greater challenge is computing the 
$\lexp$-image of a generic power of $y$.
To this end, we must first consider the $w$-expansion of $y^e$ where $e \in \Z_{>0}$. Note that we can write $e= qm+r$ with $q, r \in \Z_{\ge 0}$ such that $0 \le r \le m-1$.
We adopt the following notation for the $w$-expansion of $y^e$ promised by Lemma \ref{decompose}:
\begin{equation}\label{expandypower}
y^e= y^{qm+r} = \sum_{i=0}^{q} \sum_{j=0}^{m-1} y_{im+j}^{(e)} y^j w^i,
\end{equation}
where  $y_{im+j}^{(e)} \in K(x)$ depends on $e,i,$ and $j$.
Since $w$ is monic in the variable $y$, we have
\begin{equation}\label{yt}
y_t^{(t)} = 1.
\end{equation}
Moreover, in  (\ref{expandypower}), we have $\deg_y y^jw^i = mi+j$, and so  when
$im+j > e$, it follows that $y_{im+j}^{(e)} = 0$.
More generally,  for  $s,t \in \Z_{\ge 0}$ such that $s<t$,  we define 
\begin{equation}\label{yts}
y_t^{(s)} = 0.
\end{equation}

\begin{example}
Define $m=2$ and $w= y^2+y/x+x^3$.
The $w$-expansion of $y^4$ is
$$y^4 = (-x + x^6) + ( x^{-3} - 2 x^2)y + ((x^{-2} - 2 x^3) + 2x^{-1}y)w+w^2,$$
and so
$$y_{0}^{(4)} = -x + x^6,
y_{1}^{(4)} = x^{-3} - 2 x^2,
y_{2}^{(4)} = x^{-2} - 2 x^3,
y_{3}^{(4)} = 2x^{-1},
y_{4}^{(4)} = 1.
$$
\end{example}

In the proposition below, we demonstrate the existence of an infinite sequence of polynomials in $K[x,y]$ of  increasing $y$-degree, all of whose $\lexp$-images are bounded below by a constant determined by the parameters $(m, n, w, \alpha, \beta)$. Using linear combinations of the polynomials guaranteed by the proposition below will allow us to construct such a sequence of polynomials whose $\lexp$-images increase without bound.

\begin{prop}
\label{choosevalue}
Let $\lexp$ be the map associated  
to $(m,n, w, \alpha, \beta)$ where 
$w = y^m +  \sum_{k=0}^{m-1} w_ky^k$
with $w_k \in K(x)$. 
Suppose  further we have
\begin{enumerate}
\item[(i)] $\beta > 0$;
\item[(ii)] $\lexp(w_k) > (m-k)n \alpha$ for all $1 \le k \le m-1$;
\item[(iii)] $\lexp(w_0) = mn \alpha$.
\end{enumerate}
For all $d\in \Z_{\ge 0}$, there exists $f \in K[x,y]^*$ that is monic in $y$ such that $\deg_y f = dm$ and
$$\lexp(f) \ge  (mn -m-n) \alpha  .$$
\end{prop}

\begin{proof}
We adopt the notation (\ref{expandypower}) for the $w$-expansion of $y^{qm+r}$
where $q, r \in \Z_{\ge 0}$ with  $0 \le r \le m-1$.

We first define a finite sequence $c_0, \dots, c_{dm}$ of elements of
$K[x]$ by first defining $c_{dm}=1$ and recursively working down to 
 $c_0$.
 Specifically, given $c_{dm},  \dots, c_{t+1}$,
define $c_t$ to be the unique element of $K[x]$ promised by Lemma
\ref{polyrat}
such that
\begin{equation}\label{deltac}
 v_{\infty}\left( c_{t} +
\sum_{s={t+1}}^{dm} c_sy_{t}^{(s)}\right) > 0.
\end{equation}
If we define $f\in K[x,y]$ by $$f = \sum_{t=0}^{dm} c_t y^t,$$
we can write \begin{equation}
f = \sum_{q=0}^{d} \sum_{r=0}^{m-1} c_{qm+r} y^{qm+r},
\end{equation}
where 
\begin{equation}\label{sdm}
c_s = 0 \mbox{ for } s > dm.
\end{equation}
Using (\ref{expandypower}), we express this as
\begin{eqnarray}
f & = &   \sum_{q=0}^{d}  \sum_{r=0}^{m-1}   \sum_{i=0}^{q}  \sum_{j=0}^{m-1} c_{qm+r} y_{im+j}^{(qm+r)} y^j w^i  \\
& = &  \sum_{i=0}^{d}  \sum_{j=0}^{m-1}   \sum_{q=i}^{d}   \sum_{r=0}^{m-1} c_{qm+r} y_{im+j}^{(qm+r)} y^j w^i.
\end{eqnarray}
Thus,
\begin{equation}\label{fexpand}
f=  \sum_{i=0}^{d}  \sum_{j=0}^{m-1}  a_{im+j} y^j  w^i 
\end{equation}
where $$ a_{im+j} =  \sum_{q=i}^{d}    \sum_{r=0}^{m-1} c_{qm+r} y_{im+j}^{(qm+r)} .$$
Note that when $i \le q \le d$ and $0 \le r \le m-1$, we have 
$im \le qm+r \le dm+m-1$. Thus.
 $$ a_{im+j} = \sum_{s=im}^{dm+m-1} c_{s} y_{im+j}^{(s)} .$$
From (\ref{yts}), we have that $y_{im+j}^{(s)} =0$ whenever $s<im+j$, and so this can be rewritten as
 $$ a_{im+j} = \sum_{s=im+j}^{dm+m-1} c_{s} y_{im+j}^{(s)} .$$
Since $im+j$ represents any nonnegative integer, we can replace $im+j$ by the variable $t$ to obtain
 $$ a_{t} = \sum_{s=t}^{dm+m-1} c_{s} y_{t}^{(s)} .$$
We know that  $y_{t}^{(t)} = 1$ by (\ref{yt}), and so
$$ a_{t} = c_{t} +
\sum_{s={t+1}}^{dm+m-1} c_sy_{t}^{(s)}.$$
Moreover, 
by (\ref{sdm}) we can rewrite this as
\begin{equation}\label{dim+j}
 a_{t} = c_{t} +
\sum_{s={t+1}}^{dm} c_sy_{t}^{(s)}.
\end{equation}
Thus, by (\ref{deltac}),  for  $t< dm$,  we have $v_{\infty}(a_t) \ge 1$, and 
so by Definition \ref{def:constructval}, since $\alpha$ is negative, we have 
\begin{equation}\label{at}
\lexp(a_t)  = - v_{\infty}(a_t) m \alpha \ge - m \alpha.
\end{equation}
Now we consider the images, under $\lexp$, of the terms in the  expression (\ref{fexpand}) so that we can then use the  triangle inequality to put a bound on  the image of the entire sum. Note that 
  $a_{im+j} = 0$ for $im+j > dm$,
and so we only need to consider when $im+j \le dm$.  
   For the term corresponding to $i=d$ and $j=0$ in expression  (\ref{fexpand}), we have $a_{dm}=1$, and so 
\begin{equation}\label{imagepiece1}
   \lexp(a_{dm} y^0w^d) = d \beta > 0.
   \end{equation}
For $0 \le im+j < dm$,  
we have by (\ref{at}) that
$$\lexp(a_{im+j} y^j  w^i) \ge -m \alpha+jn\alpha + i \beta.$$ Since  $\beta >0$,
$$\lexp(a_{im+j} y^j  w^i) 
\ge -m \alpha + j n \alpha,$$
and since $0 \le j \le m-1$,
\begin{equation}\label{imagepiece2}
\lexp(a_{im+j} y^j  w^i)  \ge (mn-m-n)\alpha
\end{equation}
whenever $0 \le im+j < dm$.
Using (\ref{imagepiece1}) and (\ref{imagepiece2}), we can apply the triangle inequality to (\ref{fexpand}) to conclude
$$\lexp(f)  \ge (mn-m-n)\alpha.$$
\end{proof}

In order to use this proposition to recursively construct a sequence of polynomials whose $\lexp$-images are not bounded above, we first prove a lemma that will assist us with this process.

\begin{lemma}\label{immsucclemma}
Suppose $\lexp$ is a $K$-valuation  on $L$. Suppose further that $f_0, \dots, f_d \in L$ such that    $\lexp(f_{i+1})$ is the immediate successor of $\lexp(f_{i})$ in the value group whenever $0 \le i \le d-1$.
Then given $f \in L$ such that $\lexp(f) \ge \lexp(f_0)$, there exist $\lambda_0, \dots \lambda_d\in K$ such that
$$\lexp\left( f + \sum_{i=0}^d \lambda_i f_i \right) > \lexp(f_d).$$
\end{lemma}

\begin{proof}
Consider $f \in L$ such that $\lexp(f) \ge \lexp(f_0)$. 
If $\lexp(f) > \lexp(f_d)$, then the proof follows by setting $\lambda_0 = \cdots = \lambda_d = 0$. Otherwise, $\lexp(f) \in \{ \lexp(f_0), \cdots, \lexp(f_d) \}$, and
we proceed by  recursively constructing $h_0, \dots, h_\ell \in L$  (for some index $\ell$) such that whenever $0 \le i \le \ell-1$,
\begin{enumerate}
\item[(i)] $\lexp(h_i) \in \{ \lexp(f_0), \dots, \lexp(f_d) \}$, 
\item[(ii)] $\lexp(h_{i+1}) > \lexp(h_{i})$.
\end{enumerate}
Setting $h_0 = f$, condition (i) immediately follows.
Given $h_0, \dots, h_i$ where $0 \le i \le \ell-1$, we  construct $h_{i+1}$ as follows.
Since
$\lexp(h_{i}) = \lexp(f_{j_i})$ for some $j_i$ where
$0 \le j_i \le d$,  there exists
$\lambda_{j_i} \in K$ such that
$\lexp(h_{i} + \lambda_{j_i} f_{j_i}) > \lexp (h_{i})$, and we define $h_{i+1} = h_{i} + \lambda_{j_i} f_{j_i}$, in which case $\lexp(h_{i+1}) > \lexp(h_i)$.

Since $\lexp(f_0), \dots, \lexp(f_d)$ is a sequence of immediate successors, there must exist an index $\ell$ where $\lexp(h_{\ell}) > \lexp(f_d)$.
The conclusion follows from the observation that
$h_\ell = f + \sum_{i=0}^{\ell-1} \lambda_{j_i} f_{j_i}.$
\end{proof}

\begin{example}\label{ex:not principal}
There exists a  unique $K$-valuation  $\lexp : K(x,y) \to (\Z \oplus \Z, \text{lex})$ where $\lexp(x) = (-2,-2)$,
$\lexp(y) = (-3,-3)$, and
$\lexp(y^2+y/x + x^3) =(0,-1)$
such that $\lexp(K[x,y]^*)$ is nonpositive, yet not reversely well ordered. Moreover,
\begin{equation}\label{valmonoid}
\lexp(K[x,y]^*) = \{ (0,0)\} \cup  \left( \Z_{> 0} (-1,-1) + \Z_{\ge 0}(0,1) \right).
\end{equation}
\end{example}

\begin{proof}
Define $m=2, n=3, \alpha = (-1,-1), \beta = (0,1)$,
and  $$w= y^2+y/x+x^3.$$
By Proposition \ref{homomorphism:powers},  the map associated to $(m,n, w, \alpha, \beta)$,  is a $K$-valuation on $K(x,y)$. 
By Definition \ref{def:constructval}, for any $f \in K[x,y]^*$,
 \begin{equation}\label{lexpf}
 \lexp(f) =  \min_{0 \le i \le \ell \atop 0 \le j \le m-1} \left\{- v_{\infty} (f_{ij})m (-1,-1) + jn (-1,-1) + i  (0,1) \right\} \in \Z (-1,-1) + \Z_{\ge 0}(0,1),
\end{equation}
where $f$ has $w$-expansion $\sum_{i=0}^\ell \sum_{j=0}^{m-1} f_{i,j} y^jw^i.$
Moreover, 
$\lexp(x) = (-2,-2)$, $\lexp(y) = (-3,-3)$, and 
\begin{equation}\label{y2x3}
\lexp(y^2+x^3) = 
\lexp(w - y/x) = \min\{(0,1),(-1,-1) \} =
(-1,-1).
\end{equation}

First, we show that the valuation is nonpositive on $K[x,y]^*$. To do so, 
we not only  consider the lexicographic order on the value group $\Z \oplus \Z$, but we also endow the polynomial ring $K[x,y]^*$ with a different order, namely, the lexicographical order with $y>x^i$ for all $i \in \Z_{\ge 0}$.  
 This is a total order on the set of nonzero monomials given by the rule $x^{a_1}y^{b_1} > x^{a_2}y^{b_2}$ whenever 
either (i) $b_1-b_2$ is positive or (ii) $b_1=b_2$ and $a_1-a_2$ is positive.
Given a polynomial $f$, the leading term $\text{lt}(f)$ is defined as the term whose underlying monomial is maximal among those appearing in $f$.

Given a nonconstant $f \in K[x,y]^*$, we  consider four cases:
\begin{enumerate}
\item[(1)]  $\lt (f) = \lambda y^{2 \ell}$ with $\lambda \in K^*$,  $\ell > 0$;
\item[(2)] 
 $\lt (f) = \lambda x^k y^{2\ell}$ with $\lambda \in K^*$,  $k, \ell >0$;
\item[(3)]  $\lt (f) = \lambda x^k y^{2 \ell+1}$ with $\lambda \in K^*$, $k, \ell \ge 0$;
\item[(4)]  $\lt (f) = \lambda x^k$ with $\lambda \in K^*$,  $k  > 0$.
\end{enumerate}
Before considering these cases, we first observe that 
\begin{equation}\label{eq:w-ell}
w^ \ell  =  (y^2 + (y/x) + x^3)^\ell  =  y^{2\ell} + \ell y^{2 \ell -1} (1/x) +   g_1,
\end{equation}
where $g_1 \in K(x)[y]$ with $\deg_y (g_1) \le 2 \ell -2$.

\medskip

\noindent
{\bf Case 1:} \ 
Suppose $\lt (f) = \lambda y^{2 \ell}$ with $\lambda \in K^*$,  $\ell > 0$.
Since $\deg_y(w) = 2$, the
 $w$-expansion of $f$ is of the form $$f= \lambda w^\ell +  \sum_{i=0}^{\ell-1}   (  f_{i,0} + f_{i,1}y ) w^i,$$
and so by (\ref{eq:w-ell}),
\begin{eqnarray*}
f  &= &  \lambda y^{2\ell} + \lambda \ell y^{2 \ell -1} (1/x) +   \lambda g_1  +   \sum_{i=0}^{\ell-1}   (  f_{i,0} + f_{i,1}y ) w^i \\
 &= &  \lambda y^{2\ell} +  \left( (\lambda\ell /x)  + f_{\ell-1,1} \right) y^{2\ell-1} + 
g_2,
 \end{eqnarray*}
where $g_2 \in K(x)[y]$ with $\deg_y (g_2) \le 2 \ell -2$. 
Defining 
\begin{equation}\label{g3}
g_3 = (\lambda\ell/x)  + f_{\ell-1,1},
\end{equation}
we see that $g_3 \in K[x]$ since it
is the coefficient of $y^{2\ell-1}$ appearing in the polynomial $f \in K[x,y]$.

Now, $\lexp(-\lambda \ell/x) = (2,2)$ and $\lexp(g_3) = \deg_x(g_3) \cdot (-2,-2)$, and so $\lexp(-\lambda\ell/x) \neq \lexp(g_3)$. Therefore, using (\ref{g3}) we see
 $$\lexp(f_{\ell-1,1}) = \lexp(-\lambda\ell/x + g_3)
= \min\{ \lexp(-\lambda\ell/x) , \lexp(g_3)
\} = 
 \deg_x(g_3) \cdot (-2,-2),$$
and so $$\lexp(  f_{\ell-1,1}y ) = \lexp(f_{\ell-1,1}) + \lexp(y) 
= \deg_x(g_3) \cdot (-2,-2) + (-3,-3).$$
Since 
$f_{\ell-1,0} \in K(x)$, we have
$\lexp(f_{\ell-1,0}) \in 2 (\Z \oplus \Z)$, from which it follows that
$$\lexp(f_{\ell-1,0}) \neq \lexp(f_{\ell-1,1}y),$$
and so 
$$\lexp(f_{\ell-1,0} + f_{\ell-1,1}y) = \min\{\lexp(f_{\ell-1,0}), \lexp(f_{\ell-1,1}y)
\}
\le  \lexp(f_{\ell-1,1}y) = \deg_x(g_3) \cdot (-2,-2) + (-3,-3).$$
Thus,
\begin{eqnarray*}
\lexp((f_{\ell-1,0} + f_{\ell-1,1}y)w^{\ell-1})  & \le & \deg_x(g_3) \cdot (-2,-2) + (-3,-3)+ 
 (\ell-1)\beta\\
 & = & \deg_x(g_3) \cdot (-2,-2) + (-3,-3)+ 
 (\ell-1) (0,1)\\
 & = &
(-2 \deg_x(g_3) - 3, -2 \deg_x(g_3) - 3 + \ell -1)\\
& < & (0,0).
\end{eqnarray*}
By (\ref{prop:lexpconstruction1}), it follows that  $\lexp(f) < (0,0)$.

\medskip

\noindent
{\bf Case 2:} \ If $\lt (f) = \lambda x^k y^{2\ell}$ with $\lambda \in K^*$,  $k, \ell > 0$,  then we can write $f = py^{2 \ell}$ where
$p \in K[x]$ such that $\deg_x p = k$. By (\ref{eq:w-ell}),  the $w$-expansion of $f$ is $$f= p  w^\ell +  \sum_{i=0}^{\ell-1}   (  f_{i,0} + f_{i,1}y ) w^i,$$
and so by  (\ref{prop:lexpconstruction1}), we have 
$$\lexp(f)  \le   \lexp(p) + \ell \lexp(w) = k(-2,-2) + \ell (0,1) =  (-2k,-2k+\ell).$$
Since $k \ge 1$, it follows that
$$\lexp(f)  \le    (-2,\ell-2)< (0,0).$$

\medskip

\noindent
{\bf Case 3:} \ If $\lt (f) = \lambda x^k y^{2\ell+1}$ with $\lambda \in K^*$,   $k, \ell \ge 0$, then we can write $f = py^{2 \ell +1}$ where
$p \in K[x]$ such that $\deg_x p = k$.
By (\ref{eq:w-ell}),  the $w$-expansion of $f$ is $$f=  \sum_{i=0}^{\ell}   (  f_{i,0} + f_{i,1}y ) w^i,$$
where $f_{\ell,1} = p$,
and so by (\ref{prop:lexpconstruction1}), we have 
$$\lexp(f)  \le   \lexp(p) + \lexp(y) + \ell \lexp(w) = k (-2,-2) + (-3,-3) + \ell (0,1) = (-2k-3,-2k-3+\ell) \le (0,0).
$$

\medskip

\noindent
{\bf Case 4:} \ Suppose
$\lt (f) = \lambda x^{k}$ with $\lambda \in K^*$,  $k > 0$. 
Since we are using the lexicographical order with $y>x$, it follows that $f \in K[x]$, and so $\lexp(f) \le (0,0).$

We have seen in all four cases that $\lexp(f) \le (0,0)$, and so $\lexp(K[x,y]^*)$ is nonpositive. Next, we demonstrate that $\lexp(K[x,y]^*)$ is not reversely well ordered.
We will inductively define $f_d \in K[x,y]^*$ with $d \in \Z_{\ge 0}$ such that 
\begin{equation}\label{degf}
\deg_y f_d = 2(d+1),
\end{equation} and
\begin{equation}\label{expf}
\lexp(f_d)= (-1,d-1),
\end{equation}
 from which it follows that $\lexp(K[x,y]^*)$ is not reversely well ordered.
Defining $f_0 =  y^2+x^3$, it follows from (\ref{y2x3}) that $\lexp(f_0) = (-1,-1)$.
Given $f_0, f_1, \dots, f_d$, we will show how to construct $f_{d+1}$.

By  Proposition \ref{choosevalue}, there exists $f \in K[x,y]$ such that 
$\deg_y f = 2(d+1)$ and $\lexp (f) \ge (-1,-1) = \lexp(f_0)$.
By Lemma \ref{immsucclemma}, there exist $\lambda_0, \dots, \lambda_d \in K$
such that $$\lexp\left( f + \sum_{i=0}^d \lambda_i f_i \right) > \lexp(f_d).$$
If we define $f_{d+1} = f + \sum_{i=0}^d \lambda_i f_i$, 
then since $\deg_y f = 2(d+1)$ and $\deg_y f_i = 2i$, we can conclude (\ref{degf}).
Moreover,
$$\lexp(f_{d+1}) > \lexp(f_{d} ) =(-1,d-1) ,$$
from which it follows 
\begin{equation}
\lexp(f_{d+1}) \ge (-1,d).
\end{equation} 
Since we are working over $(\Z \oplus \Z, \mbox{lex})$ and  $\lexp(K[x,y]^*)$ is nonpositive, it follows from (\ref{lexpf}) that 
\begin{equation}\label{fded}
\lexp(f_{d+1}) = (-1,e_d)
\end{equation} for some integer $e_d  \ge d$.
 To justify (\ref{expf}), we have only left to show that $e_d \le d$.

We begin with the observation that
$$\lexp(xy^2+y+x^4) = \lexp(xw) = (-2,-2) + (0,1) = (-2,-1).$$
For each $q \in \Z_{\ge 0}$, $r \in \{0,1\}$,  we define $h_{2q+r}  \in K[x,y]$  by
$$h_{2q+r} = y^r(xy^2+y+x^4)^q,$$
in which case 
\begin{equation}\label{h2qr}
\lexp(h_{2q+r}) = r(-3,-3) + q (-2,-1).
\end{equation}
Note that for $i\neq j$, $\overline{\lexp(h_i)} \neq \overline{\lexp(h_j)}$, and since $\deg_yh_i = i$, we have by Lemma \ref{lemma:gradual-alt},
$$\lexp(\V_{2(d+1)}^*)/\lexp(\V_0^*) = \{\overline{h_0}, \overline{h_1}, \dots, \overline{h_{2(d+1)}}\}.$$
Since  $\deg_y{f_{d+1}}  = 2(d+1)$, it follows that
$\overline{\lexp(f_{d+1})} = \overline{\lexp(h_{2q+r})}$ for some nonnegative index $2q+r \le 2(d+1)$.
Since $\lexp(x) = (-2,-2)$, it follows that $\lexp(V_0^*) = (-2,-2) \Z_{\ge 0}$,
and so for some $t \in \Z$,
$$
\lexp(f_{d+1}) = \lexp(h_{2q+r}) +t (-2,-2).
$$
Combining this fact with (\ref{fded}) and (\ref{h2qr}), we have
$$(-1,e_d) =r(-3,-3) + q(-2,-1) +  t(-2,-2).$$
Subtracting the first component from the second component simultaneously on the left- and right-hand side of this equation, we find that $$e_d +1 =  (-3r+3r) + (-q+2q)+  (-2t+2t),$$ and so $e_d+1 = q$.
However,  $2q+r \le 2(d+1)$, and so $e_d+1 = q \le d+1 - \frac{r}{2}$, which implies
that $-e_d \ge - d + \frac{r}{2}$. Since $r \ge 0$, it follows that 
$e_d \le d$, as desired.

Now that we have demonstrated that $\lexp(K[x,y]^*)$ is reversely well ordered, we will precisely describe this set. Considering (\ref{lexpf}) in conjuction with the fact that $\lexp(K[x,y]^*)$ is nonpositive yields
$$\lexp(K[x,y]^*) \subset \left( \Z_{> 0} (-1,-1) + \Z_{\ge 0}(0,1)\right) \cup  \{ (0,0)\}.$$
Noting that $(0,0) \in \lexp(K[x,y])$, we justify the reverse inclusion by taking an arbitrary nonzero ordered pair of the form 
$i(-1,-1) + j(0,1) = (-i,j-i),$
with $i>0, j \ge0$ and seeing that it can be written
$$\lexp(f_j (y^2+x^3)^{i-1})  = (-1,j-1)+ (i-1)(-1,-1) = (-i,j-i),$$
by using (\ref{y2x3}) and (\ref{expf}) together.\end{proof}

\end{document}